\documentclass{amsart} 
\usepackage{amssymb}
\usepackage{amsmath}
\usepackage{amsthm}
\usepackage{amscd}
\usepackage{mathrsfs}
\usepackage{bm}
\usepackage{hyperref}

\newtheorem{Th}{Theorem}[section]

\newtheorem{Prop}[Th]{Proposition}

\newtheorem{Ques}{Question}
\newtheorem*{theorem*}{Theorem}

\newcommand{\Tr}{{\mathrm{Tr}}}
\newcommand{\tr}{{\mathrm{tr}}}
\newcommand{\E}{{\mathbb{E}}}
\newcommand{\M}{{\mathbb{M}}}

\newcommand{\R}{{\mathbb{R}}}
\newcommand{\A}{{\mathcal{A}}}
\newcommand{\Un}{{\mathcal{U}_n}}
\newcommand{\TL}{\text{TL}}
\newcommand{\Wg}{\text{Wg}}
\newcommand{\Gr}{\text{Gr}}
\newcommand{\Fix}{\text{Fix}}

\numberwithin{equation}{section}

\begin{document}
 
\title{Moment Methods on compact groups: Weingarten calculus and its applications.}

\author{Beno\^\i{}t Collins}

\address{Mathematics Department, Kyoto University, Japan} \email{collins@math.kyoto-u.ac.jp}

\begin{abstract}
A fundamental property of compact groups and compact quantum groups is 
the existence and uniqueness of a left and right invariant probability -- the Haar measure. 
This is a natural playground for classical and quantum probability,
provided it is possible to compute its moments. 
Weingarten calculus addresses this question in a systematic way.
The purpose of this manuscript is to survey recent developments,
describe  some salient theoretical properties of Weingarten functions, 
as well as
 applications of this calculus to random matrix theory, quantum probability, and algebra,
  mathematical physics and operator algebras.
\end{abstract}

\maketitle

\section{Introduction}
One of the key properties of a compact group $G$ is that it admits a unique left and right invariant probability measure
$\mu_G$. It is called the \emph{Haar measure}, and we refer to \cite{MR2098271} for reference.
In other words, $\mu_G(G)=1$, and
for any Borel subset $A$ of $G$ and $g\in G$,
$\mu_G(Ag)=\mu_G(gA)=\mu_G(A)$, where $Ag=\{hg, h\in A\}$ and $gA=\{gh, h\in A\}$.
The left and right invariance together with the uniqueness of $\mu_G$ readily imply that $\mu_G (A^{-1})=\mu_G(A)$.
The standard proofs of the existence of the Haar measure are not constructive. 
In the more general context of locally compact groups, a left (resp. right) invariant measure exists too.
It is finite if and only if the group is compact and uniqueness is up to a non-negative scalar multiple. 
In addition, the left and right Haar measures need not be the same. 
For locally compact groups, a classical proof of existence imitates the construction of the 
Lebesgue measure on $\R$ and resorts to outer measures. 
In the specific case of compact groups, a fixed point argument can be applied.
Either way, in both cases, the proof of existence is not constructive, in the sense that it does not tell us:
\emph{how to integrate functions}?
Weingarten calculus is about addressing this problem systematically.
\emph{Which functions} one wants to integrate needs, of course, to be clarified. 
We focus on the case of matrix groups, for which
there are very natural candidates: polynomials in coordinate functions.

We recast this problem as the question of \emph{computing the moments} of the Haar measure. 
Recall that
for a real random variable $X$, its moments are by definition the sequence $\E(X^k), k\ge 0$ -- whenever they are
defined. 
If the variable is vector valued in $\mathbb{R}^n$, i.e. $X=(X_1,\ldots , X_n)$, then the moments
are the numbers $\E(X_1^{k_1}\ldots X_n^{k_n}), k_1,\ldots k_n\ge 0$.
Naturally, the existence of moments is not granted and is subject to the integrability of the functions.  
In the case of matrix compact groups, we have $G\subset \M_n(\mathbb{C})=\R^{2n^2}$ therefore, we may 
consider that the random variable we are studying is a random vector in $\R^{2n^2}$ whose distribution is the
Haar measure with respect to the above inclusion.
In this sense, we are really considering a moment problem. 
For this reason, we do not consider only coordinate functions but also
their complex conjugates in our moment problem.

The goal of this note is to provide an account of Weingarten calculus and, in particular, its multiple applications,
with emphasis on the moment aspects and applications. 
From the point of view of the theory, there have been many approaches to computing integrals of functions with
respect to the Haar measure. We enumerate here a few important ones.
\begin{enumerate}
 \item Historically, the first non-trivial functions computed are arguably Fourier transform, e.g., the 
 Harish-Chandra integral,
 \cite{MR84104}.
 The literature is vast and started from the initial papers of Harish-Chandra and Itzykson Zuber until now; however, we do not elaborate too much on this field as we focus on polynomial integrals. 
 These techniques involve representation theory, symplectic geometry, and complex analysis. 
 We refer to \cite{MR3900830} for a recent approach and to the bibliography therein for references.
 \item Geometric techniques are natural because the measure can be described locally with differential geometry when compact groups are manifolds. 
They are efficient for small groups. We refer, for example, to  \cite{MR2402345} 
for such techniques and gaussianization methods, with application to quantum groups. 
Geometry is also helpful to compute specific functions, such as polynomials in one row or column
with respect to orthogonal or unitary groups. 
\item Probability, changes of variables, and stochastic calculus are natural tools to try to compute the moments of Haar measures. For example, Rains in \cite{MR1468398}
used Brownian motion on compact groups and the fact that the Haar measure is the unique invariant measure to compute a complete set of relations.
Subsequently, L\'evy, Dahlqvist, Kemp, and the author have made progress on understanding the unitary multiplicative Brownian version of Weingarten calculus in \cite{MR2407946,MR3748321}.
\item Representation theory has always been ubiquitous in the quest for calculating the Haar measure. 
A first significant set of applications can be found by \cite{MR1274717}, but results were already available by
\cite{MR593129,MR471696,MR456111}.
\item Combinatorial interpretations of the Haar measure in some specific cases were initiated in \cite{MR1411614}. 
In another direction, there were the notable works of \cite{MR1411614}.
Subsequently, new combinatorial techniques were developed in 
 \cite{MR1959915,MR2531371}, and we refer to
\cite{2010.13661} for substantial generalizations. 
We also refer \cite{MR4011702} 
for modern interpretations and applications to geometric group theory.
\end{enumerate}

As for the applications, they can be found in a considerable amount of areas, including:
theoretical physics (2D quantum gravity, matrix integrals, random tensors), mathematical physics (quantum information theory, 
Quantum spin chains), operator algebras (free probability), probability (limit theorems), representation theory, statistics, finance, machine learning, group theory.
The foundations of Weingarten calculus, as well as its applications, keep expanding rapidly, and this
manuscript is a subjective snapshot of the state of the art.
This introduction is followed by section 2 that contains the foundations and theoretical results about the Weingarten functions. Section 3 investigates `simple' asymptotics of Weingarten functions and applications to Random Matrix 
theory.
Section 4 deals with `higher order' asymptotics and applications to mathematical physics. 
Section 5 considers `uniform' asymptotics and applications to functional analysis, whereas the last section 
contains concluding remarks and perspectives.

\section{Weingarten calculus}

\subsection{Notation}

On the complex matrix algebra $\M_n(\mathbb{C})$, we denote by $\overline A$ the entrywise conjugate of a matrix $A$
and $A^*=\overline A^t$ the adjoint. 
In the sequel, we work with a compact matrix group $G$, i.e., a subgroup of $GL_n(\mathbb{C})$ of invertible
complex matrices that is compact for the induced topology. It is known that such a group is conjugate inside $GL_n(\mathbb{C})$ to the unitary group
$\Un=\{U, UU^*=U^*U=1_n\}$.
Writing an element $U$ of $\Un$ as a matrix $U=(u_{ij})_{i,j\in\{1,\ldots ,n\}}$, we view the entries $u_{ij}$
as polynomial functions $\Un\to\mathbb{C}$.
As functions, they form a $*$-algebra --  the $*$-operation being the complex conjugation. By construction, they
are separating for $\Un$, therefore, by 
Weierstrass' theorem, the $*$-algebra generated by $u_{ij},i,j\in\{1,\ldots ,n\}$, which is the algebra of polynomial
functions on $\Un$ is a dense subalgebra for the sup norm in the algebra of continuous functions on $G$.

By Riesz' theorem, understanding the Haar measure boils down to understanding
$\int_{U\in G} f(U)d\mu_G(U)$ for any continuous function. By density and linearity, it is actually 
enough to be able to calculate 
systematically 
$$\int_{U\in G} u_{i_1j_1}\ldots u_{i_kj_k}\overline{u_{i_1'j_1'}\ldots u_{i_{k'}'j_{k'}'}}d\mu_G(U).$$ 
No answer was known in full generality until a systematic development was initiated in 
\cite{MR1959915,MR2217291}.
However, in the particular case of of $\Un, \mathcal{O}_n$, 
an algorithm to calculate a development in large $n$ was devised in \cite{MR456111,MR471696}, with further 
improvements by \cite{MR758417}, and character expansions were obtained in \cite{MR593129}, 
however these approaches are largely independent.
Likewise, Woronowicz obtained a formula for the moments of characters in the case of quantum groups in 
\cite{MR901157}.
Interestingly, motivated by probability questions, the same formula was rediscovered independently by Diaconis-Shashahani \cite{MR1274717}
in the particular case of compact matrix groups.

\subsection{Fundamental formula}

Although the partial answers to the question of computing moments were rather involved, the general answer turns out, 
in hindsight, to be surprisingly simple, so we describe it here. We also refer to
\cite{2109.14890} for an invitation to the theory.
We first start with the following notation: for an element $U=(u_{ij})\in G\subset \M_n(\mathbb{C})$,  $\overline{U}$ is the entry-wise conjugate, i.e.  $\overline{U}=(\overline{U_{ij}})$. Since $U$ is unitary, $\overline{U}$ is unitary, too.
We denote by $V=\mathbb{C}^n$ the fundamental representation of $G$, 
and
$\overline{V}$ the contragredient representation. 
For a general representation $W$ of $G$, $\Fix (G,W)$ is the vector subspace of $W$ of fixed points under the action of
$G$, i.e. $\Fix (G,W)=\{x\in W, \forall U\in G, Ux=x\}$.
Finally, we fix two integers $k,k'$, and set
$$Z_G=\int_{U\in G} U^{\otimes k}\otimes \overline{U}^{\otimes k'}d\mu_G(U),$$
and abbreviate $\Fix (G,V^{\otimes k}\otimes\overline{V}^{\otimes k'})$ into 
$\Fix (G,k,k')$.

\begin{Prop}
The matrix $Z_G$ is the orthogonal projection onto $\Fix (G,k,k')$.
\end{Prop}

\begin{proof}
Since the distribution of $U$ and $UU'$ is the same for any fixed 
$U'\in G$, it implies that for any $U\in G$, $Z_G=Z_G\cdot U^{\otimes k}\otimes \overline{U^{\otimes k'}}$.
Integrating once more over $U$ gives the fact that $Z_G$ is a projection.
The fact that the map $U\to U^{-1}=U^*$ preserves the Haar measure
implies that $Z_G=Z_G^*$.
From the definition of invariance, for $x\in \Fix (G,k,k')$ and for any $U\in G$ one has
$U^{\otimes k}\otimes \overline{U^{\otimes k'}} \cdot x=x$.
Integrating with respect to the Haar measure of $G$ gives
$Z_G\cdot x=x$.
Finally, take $x$ outside $\Fix (G,k,k')$. It means that there exists $U$ such that 
$$U^{\otimes k}\otimes \overline{U^{\otimes k'}} x\ne x .$$
However, $||U^{\otimes k}\otimes \overline{U^{\otimes k'}} x||_2=|| x||_2$.
Thanks to the strict convexity of the Euclidean ball, after averaging over the Haar measure we 
necessarily get $||Z_Gx||_2<||x||_2$, which implies that $x$ is not in $\mathrm{Im} (Z_G)$.
Therefore we proved that $\mathrm{Im} (Z_G)=\Fix (G, k,k')$.
\end{proof}
From this, we can deduce an integration formula as soon as we have a 
generating family $y_1,\ldots , y_l$
 for $\Fix (G,k,k')$ (for any $k,k'$).
Let 
$$\Gr=(g_{ij})_{i,j\in\{1,\ldots , l\}}$$ be its Gram matrix, i.e.
$g_{ij}=\langle y_i,y_j\rangle$ and $W=(w_{ij})$ the pseudo-inverse of $\Gr$.
Let $E_1,\ldots , E_n$ be the canonical orthonormal basis of $V=\mathbb{C}^n$.
Let $k$ be a number and we consider the tensor space $V^{\otimes k}$ 
with its canonical orthogonal basis $E_I=e_{i_1}\otimes\ldots \otimes e_{i_k}$,
where $I=(i_1,\ldots ,i_k)$ is a multi index in $\{1,\ldots n\}^k$.
Let $I=(i_1,\ldots ,i_k, i_1',\ldots ,i_{k'}')$, $J=(j_1,\ldots ,j_k,j_1',\ldots ,j_{k'}')$ be 
$k+k'$-indices, i.e. elements of $\{1,\ldots ,n\}^{k+k'}$.
Then
\begin{Th}\label{th-general}
$$\int_{U\in G} u_{i_1j_1}\ldots u_{i_kj_k}\overline{u_{i_1'j_1'}\ldots u_{i_{k'}'j_{k'}'}}d\mu_G(U)=
\langle Z_G ,E_I\otimes E_J\rangle
=\sum_{i,j\in\{1,\ldots ,l\}}\langle E_I,y_i\rangle \langle  y_j,E_J \rangle w_{ij}$$
\end{Th}

\subsection{Examples with classical groups}
For interesting applications to be derived, the following conditions must be met:
\begin{enumerate}
\item
$y_1,\ldots , y_l$ must be easy to describe.
\item
 $\Gr$ should be easy to compute -- and if possible, its inverse, the Weingarten matrix too.
\item
$\langle E_I,y_i\rangle$ should be easy to compute.
\end{enumerate}
Let us describe some fundamental examples. 
Let $P_2(k)$ be the collection of \emph{pair partitions} on $\{1,\ldots , k\}$
($P_2(k)$ is empty if $k$ is odd, and its cardinal is $1\cdot3\cdot \ldots \cdot (k-1)=k!!$  if $k$ is even).
Typically, a partition $\pi\in P_2(k)$ consists of $k/2$ blocks of cardinal $2$,
$\pi=\{V_1,\ldots , V_{k/2}\}$, and we call
$\delta_{\pi, I}$ the multi-index Kronecker function whose value is $1$ if, for any block $V=\{k<k'\}$ of $\pi$,
$i_k=i_{k'}$, and zero in all other cases.
Likewise, 
we call $E_{\pi}=\sum_I E_I\delta_{\pi, I}$.

In \cite{MR2217291}, we obtained a complete solution to computing moments of Haar integrals for 
$\mathcal{O}_n,\Un,\mathcal{S}p_n$. 
The following theorem describes this method. For convenience, we stick to the case of $\mathcal{O}_n,\Un$.
\begin{Th}\label{th-un-on}
The entries of $\Gr$ are $\langle E_{\pi},E_{\pi'} \rangle=n^{\mathrm{loops} ( \pi,\pi')}$, and
we have $\langle E_I, E_{\pi}\rangle =\delta_{\pi, I}$.
\begin{itemize}
\item
The orthogonal case:
For $\mathcal{O}_n$, $E_{\pi}, \pi\in P_2(k)$ is a generating family of the image of $Z_{\mathcal{O}_n}$
\item
The unitary case: Thanks to commutativity and setting $2k'=k$, we consider the subset of $P_2(k)$ of pair partitions such that each block pairs one of the first $k'$ elements with one of the last $k'$ elements. 
This set is in natural bijection with the permutations $S_{k'}$, and it is the generating family of the image of $Z_{\Un}$.
\end{itemize}
\end{Th}

\begin{proof}
The first two points are direct calculations. 
The last two points are a reformulation of Schur-Weyl duality, respectively, in the case of the unitary group and of the orthogonal group (see, e.g., \cite{MR2522486}). 
\end{proof}

\subsection{Example with Quantum groups}
We finish the general theory of Weingarten calculus with a quick excursion through 
 compact matrix quantum groups. 
For the theory of compact quantum groups, we refer to  \cite{MR901157,MR943923}.
The  subtlety for quantum groups is that in general, we can not
capture all representations with just $U^{\otimes k}\otimes \overline{U^{\otimes k'}}$ because 
$U$ and $\overline{U}$ fail to commute in general. 
The theory of Tannaka-Krein duality for compact quantum groups is completely developed, and
in order to get a completely general formula, we must
instead consider 
$U^{\otimes k_1}\otimes \overline{U^{\otimes k_1'}}\otimes\ldots \otimes 
U^{\otimes k_p}\otimes \overline{U^{\otimes k_p'}}$.

Let us just illustrate the theory with the \emph{free quantum orthogonal group} 
$O_n^+$. It was introduced by Wang in \cite{MR1316765},
and its Tannaka-Krein dual was computed by Banica in \cite{MR1484551}.
Its algebra of polynomial functions
$\mathbb{C} (O_n^+)$ is the non-commutative unital $*$-algebra generated by $n^2$ self-adjoint elements $u_{ij}$ that satisfy the relation
$\sum_{k}u_{ik}u_{jk}=\delta_{ij}1$ and $\sum_{k}u_{ki}u_{kj}=\delta_{ij}1$.
Note that the abelianized version of this  unital $*$-algebra is the $*$-algebra of polynomial functions on $\mathcal{O}_n$,
which explains why it is called the free orthogonal quantum group.
There exists a unital $*$-algebra homomorphism, called the coproduct
$\Delta: \mathbb{C} (O_n^+)\to \mathbb{C} (O_n^+)\otimes \mathbb{C} (O_n^+)$
defined on generators by $\Delta u_{ij}=\sum_k u_{ik}\otimes u_{kj}$,
and a unique linear functional $\mu : \mathbb{C} (O_n^+)\to \mathbb{C}$ such that $\mu (1)=1$ and
$$(\mu\otimes Id)\Delta = 1\mu, (Id\otimes \mu)\Delta = 1\mu .$$
This functional is known as the Haar state, and it extends the notion of Haar measure on compact groups. 
Although the whole definition is completely algebraic, the proofs rely on functional analysis and operator 
algebras. 

However, the calculation of the Haar state is purely algebraic and just relies on the notion of non-crossing pair 
partitions, denoted by $NC_2(k)$, which are a subset of $P_2(k)$ defined as follows. 
A partition $\pi$ of $P_2(k)$ is non-crossing -- and therefore in $NC_2(k)$ if any two of its blocks $\{i, j \}$ and 
$\{i', j' \}$ fail to satisfy the crossing relations $i<j, i'<j', i<i'<j<j'$.
This notion was found to be of crucial use for free probability by Speicher, see, e.g. \cite{MR2266879}.
The following theorem is a particular case of a series of results that can be found in \cite{MR2341011}:
\begin{Th}
In the case of $O_n^+$,  for $U^{\otimes k}$,
the complete solution follows from the following result: 
$E_{\pi}, \pi\in NC_2(k)$ is a generating family of the image of $Z_{O_n^+}$
\end{Th}
Note that  since $U=\overline U$, 
it is enough to consider $U^{\otimes k}$ to compute the Haar measure fully.
We refer to 
\cite{MR2779129,MR2274824,MR2437834} for applications of classical Weingarten functions to quantum groups and to 
\cite{MR2402345,MR2341011,MR2776622} for further developments of quantum Weingarten theory.

\subsection{Representation theoretic formulas}

A representation theoretic approach to Weingarten calculus 
is available for many families of groups, including unitary, orthogonal
 and symplectic groups. 
Here we only describe the unitary group, and for the others, we refer to \cite{MR2567222,MR3077830}. 

Call $S_k$ the symmetric group and
consider its \emph{group algebra} $\mathbb{C}[S_k]$ -- the unital $*$-algebra whose basis as a vector space
is $\lambda_{\sigma},\sigma \in S_k$, and endowed with the 
multiplication $\lambda_{\sigma}\lambda_{\tau}=\lambda_{\sigma\tau}$
and the $*$-structure $\lambda_{\sigma}^*=\lambda_{\sigma^{-1}}$.
We follow standard representation theoretic notation, see, e.g. \cite{MR2643487}
 and $\lambda\vdash k$ denotes a Young diagram $\lambda$ has $k$ boxes.
  $\lambda\vdash k$ enumerates both the conjugacy classes of $S_k$ and its irreducible representations. 
 The symmetric group $S_k$ acts on the set $\{1,\ldots ,k\}$, and in turn, by leg permutation on
$(\mathbb{C}^n)^{\otimes k}$, which induces an algebra morphism
$\mathbb{C}[S_k]\to \M_n(\mathbb{C})^{\otimes k}$.
By Schur-Weyl duality, $\lambda$ also describes irreducible polynomial representations of the unitary 
group $\Un$ if its length is less than $n$ and in this context, $V_{\lambda}$ stands for the associated representation of the unitary group.
For a permutation $\sigma\in S_k$, we call $\#\sigma$ the number of cycles (or loops) in its cycle product decomposition
(counting fixed points). 
Consider the function 
$$G=\sum_{\sigma \in S_k} n^{\# \sigma}\lambda_{\sigma},$$
and its pseudo-inverse $W=G^{-1}=\sum_{\sigma \in S_k} w (\sigma)\lambda_{\sigma}$.
The following result was observed by the author and \'Sniady in \cite{MR2217291}
and it provides the link between representation theory and Weingarten calculus:

\begin{Th}
$G$ is positive in $\mathbb{C}[S_k]$. In addition, we have
$w (\sigma,\tau)=w(\tau\sigma^{-1})$, which we rename as $\Wg (n,\tau\sigma^{-1})$, 
and the following character expansion:
$$\Wg(n, \sigma)=\frac{1}{k!^2}\sum_{\lambda\vdash k} \frac{\chi_{\lambda}(e)^2\chi_{\lambda}(\sigma)}{dim V_{\lambda}}.$$
\end{Th}

\begin{proof}
Consider the action of $S_k$ on $(\mathbb{C}^n)^{\otimes k}$ by leg permutation. It extends to a 
unital $*$-algebra morphism $\phi:\mathbb{C}[S_k]\to \M_n(\mathbb{C})^{\otimes k}$. 
By inspection, for $A\in \mathbb{C}[S_k]$, $\Tr [\phi (A)]=\tau (GA)$, where $\tau$ is the regular trace
$\tau (\lambda_g)=\delta_{g,e}$. The positivity of $\tau$ implies that of $G$ which proves positivity.
The remaining points follow from the fact that $G$ is central and by a character formula. 
\end{proof}

\subsection{Combinatorial formulations}
Let us write formally $n^{-k}G=\lambda_e+\sum_{\sigma \in S_k-\{e\}} n^{\# \sigma -k}\lambda_{\sigma}$.
It follows that as a power series in $n^{-1}$,
$$n^k W=\lambda_e+\sum_{p\ge 1} (-1)^p \Big(\sum_{\sigma \in S_k-\{e\}} n^{\# \sigma -k}\lambda_{\sigma}\Big)^p$$
Reading through the coefficients of this series gives a combinatorial formula for 
$\Wg$ in the unitary case. Such formulas were first found in  \cite{MR1959915}, and we refer to
\cite{2010.13661} for substantial generalizations. 
See also \cite{MR4011702} 
for other interpretations, as well as \cite{MR1411614}. 

However, this formula is signed, and therefore impractical for the quest of uniform asymptotics.  
In a series of works of Novak and coworkers in \cite{MR3095005,MR3033628,MR3248222,MR3010693}
came with an exciting solution to this problem which we describe below. 
It relies on  Jucys Murphy elements, which are the following elements of $\mathbb{C}[S_k]$:
$J_i=\sum_{j>i}\lambda_{(ij)}.$
The following important result was observed: 
$$G=(n+J_1)\ldots (n+J_{k-1})$$
This follows from the fact that every permutation $\sigma$ has a unique factorization 
as $$\sigma = (i_1 j_1)\ldots (i_l j_l)$$
with the property $i_p<j_p$ and $j_p<j_{p+1}$.

This prompts us to define $P(\sigma, l)$ to be the set  of solutions to the equation
$\sigma = (i_1j_1)\ldots (i_lj_l)$ with $i_p<j_p$, $j_p\le j_{p+1}$.
The number of solutions to this problem is related to Hurwitz numbers; for details, we refer for example 
to \cite{2010.13661} and the above references.
From this, we have the following theorem:
\begin{Th} 
For $\sigma \in S_{k}$, we have the  
expansion
\begin{equation} \label{eq:unitary-Wg-expansion1}
\Wg (n,\sigma )= n^{-k}  \sum_{l \ge 0} \# P(\sigma, l) (-n^{-1})^{l}.
\end{equation}
\end{Th}
The first strategy to compute the Weingarten formula was intiated in \cite{MR471696}.
Let us outline it. We can write
$\Wg (n,\sigma )=\int u_{11}\ldots u_{kk}\overline{u_{1\sigma 1}\ldots u_{k\sigma k}}$. Indeed,  
when considering the integral on the right hand side in Theorems \ref{th-general} and \ref{th-un-on},
the only pairing appearing corresponds to $\Wg (n,\sigma )$.
Replacing the first row index of $u$ and
$\overline{u}$ by $i$ and summing over $i$, we are to evaluate
\begin{equation}
\begin{split}
\sum_{i=1}^n\int u_{i1}\ldots u_{kk}\overline{u_{i\sigma (1)}\ldots u_{k\sigma (k)}}\\
=\delta_{1\sigma (1)}\int u_{22}\ldots u_{kk}\overline{u_{2\sigma (2)}\ldots u_{k\sigma (k)}}
=n\Wg (n,\sigma )+\sum_{i=2}^l\Wg (n,(1 \, i)\sigma )
\end{split}
\end{equation}
where the first equality follows from orthogonality and the second from repeated 
uses of the Weingarten formula.
The second line provides an iterative technique to compute $\Wg (n,\sigma )$ numerically and
combinatorially. Historically, this is the idea of Weingarten, and in \cite{MR471696},
he proved that the collection of all relations obtained above determine uniquely 
$\Wg$ for $k$ fixed, $n$ large enough.  

In \cite{MR3680193}, we revisited his argument and figured out that these equations can
be interpreted as a fixed point problem and a path counting formula, both formally and numerically.
We got theoretical mileage from this approach and obtained new theoretical results, such as
\begin{Th} 
All unitary Weingarten functions and all their derivatives are monotone on $(k,\infty )$
\end{Th}

The unavoidability of 
Weingarten's historical argument becomes blatant when one studies 
 quantum Weingarten function. 
Partial results about their asymptotics were obtained in \cite{MR2776622},
however, the asymptotics were not optimal for all entries.
On the other hand, motivated by the study of planar algebras, Vaughan Jones asked us the following question: 
consider the canonical basis of the Temperley Lieb algebra $\TL_k(n)$, are the 
coefficients of the dual basis all non-zero when expressed in the canonical basis?
For notations, we refer to our paper \cite{MR3874001}.
One motivation for this question is that the dual element of the identity is a multiple
of the Jones-Wenzl projection. 

Observing that this question is equivalent, up to a global factor, to the problem
of computing the Weingarten function for $O_n^+$, and realizing that
representation theory did not give tractable formulas in this case,
we revisited the original idea of Weingarten and proved the following result, answering a series
of open questions of Jones:
\begin{Th}
The quantum $O_n^+$ Weingarten function is never zero on the non-critical interval $[2,\infty )$, and monotone. 
\end{Th}
Our proof provides explicit formulas for a Laurent expansion of the free $\Wg$ in the neighborhood of $n=\infty$,
as a generating series of paths on graphs.

\section{Asymptotics and properties of Weingarten functions}

In this section, we are interested in the following problem. For a given permutation $\sigma\in S_k$,
what is the behavior as $n\to\infty$ of
$\Wg (n, \sigma )$? This function is rational as soon as $n\ge k$, and 
even elementary observations about its asymptotics have non-trivial
applications in analysis. In the forthcoming subsections, we refine our
study of the asymptotics iteratively and derive new applications each time. 
Similar results have been obtained for most sequences of classical compact groups, but
we focus here mostly on $\Un$ and $\mathcal{O}_n$, and refer to the literature for other compact groups. 

\subsection{First order for identity Weingarten coefficients and Borel theorems}
Let us first setup notations related to \emph{non-commutative probability spaces} 
and of \emph{convergence in distribution} in a non-commutative sense.
A non-commutative probability space (NCPS) is a unital $*$-algebra $\A$ together with a state $\tau$ ($\tau : \A\to \mathbb{C}$ 
is linear, $\tau (1)=1$ and $\tau (xx^*)\ge 0$ for any $x$). 
In general we will assume \emph{traciality}: $\tau (ab)=\tau(ba)$ for all $a,b$.
 
 Assume we have a family of NCPS $(\A_n,\tau_n)$, a limiting object $(\A,\tau )$
 and a $d$-tuple $(x_n^1,\ldots ,x_n^d)\in \A_n^d$. We say that this $d$-tuple of
 \emph{non-commutative random variables converges in distribution} 
 to  $(x^1,\ldots ,x^d)\in \A^d$ iff for any sequence $ i_1,\ldots , i_k$ of indices
 in $\{1,\ldots , d\}$, 
 $$\tau_n (x_n^{i_1}\ldots x_n^{i_k})\to \tau (x^{i_1}\ldots x^{i_k})$$
 In the abelian case this corresponds to a convergence in moments (which is not in general the
 convergence in distribution), however in the non-commutative framework, it is 
usually called  convergence in non-commutative distribution, cf 
 \cite{MR1217253}.
The following result was proved in \cite{MR2408577} in the classical case and \cite{MR2341011} in the 
quantum case:
\begin{Th}
Consider a sequence of vectors $(A_1^n, \ldots , A_r^n)$ in $\M_n(\mathbb{R})$
such that the matrix $(\tr (A_iA_j^t))$ converges to $A$, and a $\mathcal{O}_n$-Haar distributed
random variable $U_n$. 
Then, as $n\to\infty$, the sequence random vectors 
$$(\Tr(A_1^nU_n ), \ldots, \Tr(A_r^nU_n ))$$
converges in moments (and in distribution) to a Gaussian real vector of covariance $A$.
If we assume instead $U_n$ to be in  $O_n^+$, then
$(\Tr(A_1^nU_n ), \ldots, \Tr(A_r^nU_n ))$
converges in non-commutative distribution to a free semicircular family of covariance $A$.
\end{Th}
The proof relies on two ingredients. 
Firstly, for all examples considered so far, 
$\Gr = n^k \cdot 1_l (1+ O(n^{-1}))$, which implies that 
$W=\Gr^{-1}=n^{-k}1_l(1+O(n^{-1}))$. By inspection, it turns out that in the above theorem, the only
entries of $W$ that contribute asymptotically are the diagonal ones,
and one can conclude with the classical (resp. the free) Wick theorem.

\subsection{Other leading orders for Weingarten coefficients}
The asymptotics obtained in the previous section are sharp only for the diagonal coefficients. However, they already yield non-trivial limit theorems. For more refined theorems, it is, however, 
necessary to obtain sharp asymptotics for all Weingarten coefficients. 
In the case of $\Un$, sharp asymptotics can be deduced from the following 
\begin{Th} 
In the case of the full cycle in $S_k$, we have the following explicit formula:
$$\Wg (n, (1\cdots k))=\frac{(-1)^{k+1}c_k}{(n-k+1)\ldots (n+k-1)},$$
where $c_k=(k+1)^{-1}{2k \choose k}$ is the Catalan number. 
In addition, $\Wg$ is almost multiplicative in the following sense:
if $\sigma$ is a disjoint product of two permutations $\sigma = \sigma_1\sqcup \sigma_2$
then 
$$\Wg (n, \sigma)=\Wg (n, \sigma_1)\Wg (n, \sigma_2)(1+O(n^{-2}))$$
\end{Th}
This result defines recursively a function $\mathrm{Moeb}:\sqcup_{k\ge 1} S_k \mapsto \mathbb{Z}-\{0\}$ satisfying
$$\Wg (n, \sigma)=n^{-k-|\sigma |} \mathrm{Moeb}(\sigma)(1+O(n^{-2})).$$
This function was actually already introduced by Biane in \cite{MR1644993}, and
it is closely related to Speicher's non-crossing M\"obius function on the incidence algebra
of the lattice of non-crossing partitions -- see e.g. \cite{MR2266879}.
Similar results are available for the orthogonal and symplectic group; we refer to \cite{MR2485426}.
Finally, let us mention that the asymptotic Weingarten function for the unitary group is the object of intense study; see for example \cite{MR4011702,2107.10252}.

\subsection{Classical Asmptotic freeness}
Weingarten calculus allows answering the following
\begin{Ques}\label{main-question}
Given two families $(A_i^{(n)})_{i\in I}$ and $(B_j^{(n)})_{j\in J}$ of matrices in $\M_n(\mathbb{C})$, what is the joint behavior of
$(A_i^{(n)})_{i\in I}\sqcup (U_nB_j^{(n)}U_n^*)_{j\in J}$, where $U_n$ is invariant according to the Haar measure on
$\Un$?
\end{Ques}
The notion of behavior has to be clarified, and it will be refined at the same time as we refine 
our estimates of the Weingarten function. 
For now, we assume that $(A_i^{(n)})_{i\in I}$ and $(B_j^{(n)})_{j\in J}$ have asymptotic moments, namely,
for any sequence $i_1, \ldots , i_l$  
$$\tr A_{i_1}^{(n)}\ldots A_{i_l}^{(n)}$$
admits a finite limit, and likewise for $(B_j^{(n)})_{j\in J}$ (note that our standing notation is  $\tr=n^{-1}\Tr$). 
In this specific context, the question becomes:
\begin{Ques}\label{ques-asymptotic moments}
Does the enlarged family
$(A_i^{(n)})_{i\in I}\sqcup (U_nB_j^{(n)}U_n^*)_{j\in J}$ have asymptotic moments?
\end{Ques}
Let us note that
since the moments are random, the question admits variants, namely, 
does the enlarged family have asymptotic moments \emph{in expectation, almost surely}?
The answer turns out to be \emph{yes} -- irrespective of the variant chosen --, and the above asymptotics allow us to deduce the joint behavior of random matrices in the large dimension.
We recall that a family of unital $*$-subalgebras $\A_i, i\in I$ of a NCPS $(\A, \tau )$ is \emph{free} iff
for any 
$l\in \mathbb{N}_*, i_1, \ldots ,i_l\in I, i_1\ne i_2, \ldots, i_{l-1}\ne i_l, \tau (x_1\ldots x_l)=0$
as soon as: (i) $\tau (x_j)=0$, and (ii) $x_j\in A_{i_j}$. 
Asymptotic freeness holds when a family has a limit distribution, and  the limiting distribution generates
free $*$-subalgebras. 

\begin{Th}
The answer to Question \ref{ques-asymptotic moments} is yes. The limit of the union is determined by
the relation of asymptotic freeness, and the convergence is almost sure. 
\end{Th}

The proof relies on calculating moments together with our knowledge of the asymptotics of the Weingarten function.
In the following theorem, we observe that different types of `asymptotic behavior', such as the
existence of a limiting point spectrum, are also preserved
under the enlargement of the family.
The theorem below is a particular case of a result to be found in  \cite{MR3830802}:
\begin{Th}
Let $\lambda_{i,n}$ be sequences of complex numbers such that $\lim_n \lambda_{i,n}=0$.
Let $\Lambda_{i,n}=diag (\lambda_{i,1},\ldots , \lambda_{i,n})$ and $A_{j,n}$ be random matrices
with the property that (i)  $(A_{j,n})_j$ converges in NC distribution as $n\to\infty$ and 
(ii) $(UA_{j,n}U^*)_j$ has the same distribution as $(A_{j,n})_j$ as a $d$-tuple of random matrices. 
Let $P$ be a non-commutative polynomial.
Then the eigenvalues of $P(\Lambda_{i,n},A_{j,n})$ converge almost surely. 
\end{Th}
The proof is also based on Weingarten calculus and moment formula. The limiting distribution 
is of a new type -- involving pure point spectrum -- and we call it \emph{cyclic monotone convergence}.

\subsection{Quantum Asymptotic freeness}

Finally, let us discuss another seemingly completely unrelated application to asymptotic representation theory.
The idea is to replace classical randomness with quantum randomness.
To keep the exposition simple, we stick to the case of the unitary group, although more general results are true for more general Lie groups, see \cite{MR2485426}.
Call $E_{ij}$ the canonical matrix entries of $\M_n(\mathbb{C})$, and $e_{ij}$ the generators of the enveloping Lie algebra $\mathfrak{U}(GL_n(\mathbb{C}))$
 of $GL_n(\mathbb{C})$, namely, the unital $*$-algebra generated by $e_{ij}$ and the relations $e_{ij}^*=e_{ji}$ and
$[e_{ij},e_{kl}]=\delta_{jk}e_{il}-\delta_{il}e_{kj}$.
The map $E_{ij}\to e_{ij}$ can be factored through all Lie algebra representations of $\Un$, and we are interested in the
following variants of its Choi matrix
$$A_n^{(1)}=\sum_{ij}E_{ij}\otimes e_{ij}\otimes 1 \,\, ,\,\,
A_n^{(2)}=\sum_{ij}E_{ij}\otimes 1\otimes e_{ij}
\in \M_n(\mathbb{C})\otimes \mathfrak{U}(GL_n(\mathbb{C}))^{\otimes 2}.$$
In \cite{MR3816511}, thanks -- among others -- to asymptotics of Weingarten functions, we proved the following, extending considerably the results of  \cite{MR1317523}.
\begin{Th}
For each $n$, take $\lambda_n,\mu_n$ two Young diagrams corresponding to a
polynomial representations $V_{\lambda_n},V_{\mu_n}$ of $GL_n(\mathbb{C})$. Assume that both dimensions
 tend to infinity as $n\to\infty$ and consider the traces on 
$\chi_{\lambda_n},\chi_{\mu_n}$ on $\mathfrak{U}(GL_n(\mathbb{C}))$.
Assume that $A_n$ converges in non-commutative distribution in Voiculescu's sense both
for $\tr\otimes \chi_{\lambda_n}$ and $\tr\otimes \chi_{\mu_n}$.
Then $A_m^{(1)},A_n^{(2)}$ are asymptotically free with respect to 
$\tr\otimes \chi_{\lambda_n}\otimes\chi_{\mu_n}$.
\end{Th}

\section{Multiplicativity and applications to Mathematical physics}

\subsection{Higher Order Freeness}

The asymptotic multiplicativity of the Weingarten function states that
$\Wg (\sigma_1\sqcup \sigma_1)=\Wg (\sigma_1 )\Wg (\sigma_2 ) (1+O(n^{-2}))$
and it is very far reaching. 
The fact that the error term $O(n^{-2})$ is summable in $n$ allows in \cite{MR1959915}
to use  a Borel-Cantelli lemma and
prove almost sure convergence of moments for random matrices; cf \cite{MR1601878} for the original proof.

A more systematic understanding of the error term is possible and has deep applications in random matrix theory.
It requires the notion of classical cumulants that we recall now.
Let $X$ be a random variable, the cumulant $C_p(X)$ is defined formally by:  
\[  C(t) = \log \E (\exp tX)= \sum_{p\ge 1}t^p\frac{C_p(X)}{p!} \; .\]
For instance, the second cumulant $ C_2(X) = \E(X^2) -  \E(X)^2 $
is the variance of the probability distribution of $X$. 
$C_p(X)$ is well defined as soon as $X$ has moments up to order $p$,
and it is an $n$-homogeneous function in $X$, therefore we can
polarize it and define an $p$-linear symmetric function
$(X_1,\ldots , X_p)\to C_p(X_1,\ldots , X_p)$. 
For any partition $\pi$ of $p$ elements with blocks $B\in \pi$, we define $C_{\pi}(X_1,\ldots , X_p)= \prod_{B\in\pi }  
C \left(\prod_{i\in B} X_i \right) $. We are now in the position to write the expectations in term of the cumulants:
$$ \E \left( \prod_{i=1}^p X_i \right) =\sum_{\pi \in P(p) } C_{\pi }.$$
The equation can be inverted through the M\"obius inversion formula.  
Asymptotic freeness considers the case where moments have a limit, whereas 
higher order asymptotic freeness
considers the case where things are known about the fluctuations of the moments: in addition to the existence
of $\lim_n\tr A_{i_1}^{(n)}\ldots A_{i_l}^{(n)}$, we assume the existence of 
$$\lim_n n^{2k-2}C_k(A_{i_{11}}^{(n)}\ldots A_{i_{l_11}}^{(n)}, \ldots , A_{i_{1k}}^{(n)}\ldots A_{i_{l_kk}}^{(n)})$$
for any sequence of indices. 
We call this set of limits the \emph{higher order limit}.
In \cite{MR2302524}, we proved
\begin{Th}
The extended family $(A_i^{(n)})_{i\in I}\sqcup (UB_j^{(n)}U^*)_{j\in J}$ admits a higher order limit.
In addition, a combinatorial rule exists to construct the joint asymptotic correlations from the asymptotic correlations of each family.
\end{Th}
This rule extends freeness and is called \emph{higher order freeness}.
Subsequent work was done in the case of orthogonal invariance by Mingo and Redelmeier. 

\subsection{Matrix integrals}

Historically, matrix integrals have been studied before higher order freeness. 
However, from the point of view of formal expansion, higher order freeness supersedes matrix integrals.
In \cite{MR1959915}, we proved the following

\begin{Th} 
Let $A$ be a non-commutative polynomial in formal variables 
$(Q_i)_{i\in I}$, formal unitaries $U_j,j\in J$ and their adjoint. 
Consider in $\M_n(\mathbb{C})$ matrices $(Q_i^{(n)})_{i\in I}$ admitting a
joint limiting distribution as $n\to\infty$, and in iid Haar distributed $(U_j^{(n)})_{j\in J}$ and their adjoint. 
Evaluating $A$ in these matrices in the obvious sense, we obtain a random matrix $A_n$ and consider
 the Taylor expansion around zero of the function
\begin{equation*}
z\rightarrow n^{-2}\log E (\exp (zn^2 A_n)) = \sum_{q\geq 1} a_q^{(n)} z^q,
\end{equation*}
Then, for all $q$, $\lim_n a_q^{(n)}$ exists and depends only on the polynomial and the limiting distribution of $Q_i^{(n)}$.
\end{Th}

In \cite{MR2531371}, we upgraded this result in the case where $A_n$ is selfadjoint and proved that there exists a \emph{real} neighborhood of zero on which the convergence holds uniformly.
The complex convergence remains a difficult problem, as a uniform understanding of the higher genus expansion must be obtained. Novak made a recent breakthrough in this direction, in the case of the HCIZ integral, see \cite{2006.04304}.

\subsection{Random Tensors}
Let us revisit Question \ref{main-question}, under the assumption that $U$ has \emph{more structure}, i.e. 
less randomness. Our model is a tensor structure, namely,
$U=U_1\otimes\ldots \otimes U_D$ where
$U_i\in \M_n(\mathbb{C})$  are iid. In other words we are interested in the symmetries under
conjugation by elements of the group $\Un^{\otimes D}$.
The joint moments of a matrix are a complete invariant of global symmetry under $\Un$-conjugation in
$\M_n(\mathbb{C})$, however for $U_1\otimes\ldots \otimes U_D$-invariance in 
$\M_n(\mathbb{C})^{\otimes D}$, one needs more invariants,
generated, for $\sigma_1,\ldots , \sigma_D\in S_k$, by
\begin{equation}
\begin{split}
\Tr_{\sigma_1,\ldots ,\sigma_D}(A)=
\sum_{i_{11},\ldots ,i_{Dk},j_{11},\ldots ,j_{Dk}}
A_{i_{11}\ldots i_{D1},j_{11}\ldots j_{D1}}\ldots 
A_{i_{1k}\ldots i_{Dk},j_{1k}\ldots j_{Dk}}\\
\delta_{i_{11},j_{1\sigma_1(1)}}\ldots \delta_{i_{1k},j_{1\sigma_1(k)}}
\cdots
\delta_{i_{D1},j_{D\sigma_D(1)}}\ldots \delta_{i_{Dk},j_{D\sigma_D(k)}}.
\end{split}
\end{equation}
 In the case of higher tensors, thanks to the Weingarten calculus, we unveil many new inequivalent asymptotic regimes 
 for higher order tensors. 
 These questions are addressed in a series of projects with Gurau and Lionni, starting with \cite{2010.13661}. 
 We study the asymptotic expansion of the
 Fourier transform of the tensor valued Haar measure -- a tensor extension of the Harish-Chandra integral
 to tensors and considerably extend the single tensor case. 
 Just as the HCIZ integral can be seen as a generating function for monotone Hurwitz numbers, which count certain 
 weighted branched coverings of the 2-sphere, the integral studied in \cite{2010.13661} leads to a generalization of monotone Hurwitz numbers, which count weighted branched coverings of a 
 collection of 2-spheres that `touch' at one common non-branch node.

\subsection{Quantum Information Theory}

Quantum Information theory has been a powerful source of problems in random matrix theory in the last two decades, and their tensor structure has made it necessary to resort to moment techniques. The goal of this section is to 
elaborate on a few salient cases.  
One starting point is the paper \cite{MR2443305} where the authors compute moments of the output of 
random quantum channels. 
We just recall here strictly necessary definitions, and refer to \cite{MR3432743} for details.
A \emph{quantum channel} $\Phi$ is a linear map $\M_n(\mathbb{C})\to \M_k(\mathbb{C})$ that preserves the non-normalized trace, and that is completely
positive, i.e. $\Phi\otimes \mathrm{Id}_l: \M_n\otimes \M_l(\mathbb{C})\to \M_k(\mathbb{C})\otimes \M_l(\mathbb{C})$ 
is positive for any integer $l$.
It follows from Stinespring theorem that for any quantum channel, there exists an integer $p$ and
an isometry $U: \mathbb{C}^n\to \mathbb{C}^k\otimes \mathbb{C}^p$ 
such that 
$\Phi (X)= (\mathrm{Id}_k\otimes \Tr_p)UXU^*$.

The set of \emph{density matrices} $D_n$ 
consists in the selfadjoint matrices whose eigenvalues are non-negative and whose trace is $1$.
For $A\in D_n$, we define its \emph{von Neumann Entropy } $H(A)$ as $\sum_{i=1}^n-\lambda_i(A)\log \lambda_i(A)$
with the convention that $0\log 0=0$ and the eigenvalues of $A$ are $\lambda_1(A)\ge\ldots \ge \lambda_n(A)$.
The \emph{minimum output entropy} of a quantum channel $\Phi$ is defined as 
$H_{min}(\Phi)=\min_{A\in D_n} H(\Phi (A))$, 
and a crucial question in QIT was whether one can find $\Phi_1,\Phi_2$ such that
$$H_{min}(\Phi_1\otimes \Phi_2)<H_{min}(\Phi_1)+H_{min}(\Phi_2)$$
For the statement, implications and background, we refer to \cite{MR3432743}.
An answer to this question was given in \cite{Hastings}
and it relies on random methods, which motivates us to consider quantum channels obtained from Haar unitaries.
A description of $\Phi (D_n)$ in some appropriate large $n$ limit has been found in \cite{MR2995183},
and the minimum in the limit of the entropy was found in \cite{MR3452274}.
In the meantime, the image under the tensor product of random channels $\Phi_1\otimes \Phi_2$ of appropriate 
matrices (known as Bell states) had to be computed.
To achieve this, we had to develop a graphical version of Weingarten calculus in \cite{MR2651902}.

We consider
the case where $k$ is a fixed integer, and
$t\in (0,1)$ is a fixed number. 
For each $n$, we consider a random unitary matrix $U \in \M_{nk}(\mathbb{C} )$, and a projection $q_n$ of 
$\M_{nk} (\mathbb{C} )$ of rank
$p_n$ such that $p_n/(nk)\sim t$ as $n\to\infty$.
Our model of a random quantum channel is 
$\Phi : \M_{p_n}(\mathbb{C} )\to \M_n(\mathbb{C})$
given by
$\Phi (X)=\tr_k (U X U^*)$,
where $\M_{p_n}(\mathbb{C} )\simeq q_n\M_{nk}(\mathbb{C} )q_n$.  
By $Bell$ we denote the Bell state on $\M_{p_n}(\mathbb{C} )^{\otimes 2}$. In \cite{MR2651902}, we proved

\begin{Th}
\label{thm:product-channel}
Almost surely, as $n \to \infty$, the random matrix $\Phi\otimes\overline\Phi(\mathrm{Bell}) \in \M_{n^2}(\mathbb{C})$ has non-zero eigenvalues converging towards
\[\gamma^{(t)} = \left( t + \frac{1-t}{k^2},\underbrace{\frac{1-t}{k^2}, \ldots, \frac{1-t}{k^2}}_{k^2-1 \text{ times}}\right).\]
\end{Th}
This result plays an important result in the understanding of phenomena underlying the sub-additivity of the 
minimum output entropy, and relies heavily on Weingarten calculus, and in particular
a graphical interpretation thereof. Much more general results in related areas of 
Quantum Information Theory have been obtained in 
\cite{MR3346125,MR2830615,MR2967961,MR3083280,MR3862490,MR3658153}.

\section{Uniform estimates and applications to analysis}
\subsection{A motivating question}

The previous sections show that when the degree of a polynomial is fixed, very precise asymptotics can be obtained in the limit of large dimension. For the purpose of analysis, an important question is, whether such estimates 
hold uniformly. 
About 20 years ago,  Gilles Pisier asked me the following question: given $k$ iid Haar unitaries $U_1^{(n)},\ldots ,U_k^{(n)}\in \Un$, what is the 
large dimension behavior of the real random variable $$t_n=||U_1^{(n)}+\ldots +U_k^{(n)}||_{\infty},$$
where $||\cdot ||_{\infty}$ stands for the operator norm?
It follows from asymptotic freeness results that almost surely $\liminf t_n\ge 2\sqrt{k-1}$ as soon as $k\ge 2$.
Setting $X_n=U_1^{(n)}+\ldots +U_k^{(n)}$, 
it would be in principle enough to estimate 
$$\E(\Tr ((X_nX_n^*)^{l(n)}))$$ for $l(n)>>\log n$. 
However, there are two significant hurdles:
(i)
Uniform estimates of Weingarten calculus would be needed. 
(ii) 
Unlike in the multi matrix model case, the combinatorics grow exponentially, and a direct 
moment approach is not possible. 
Both hurdles require developing specific tools, which we describe in the sequel. 

One notion on which we rely heavily is that of {\em strong convergence}.
Given a multi matrix model that admits a joint limiting distribution in Voiculescu's sense, 
we say that it  {\em converges strongly}
iff the operator norm of any polynomial $P$, evaluated in the matrices of the  model 
-- thus yielding the random matrix $P_n$ -- satisfies
$$\lim_n ||P_n||=\lim_\ell (\lim_n n^{-1}\Tr ((P_nP_n^*)^\ell)^{(2\ell)^{-1}}.$$
In other words, the operator norm of any matrix model obtained from a non-commutative polynomial converges
to the operator norm of the limiting object. 
Strong convergence was established in \cite{MR2183281} in the case of Gaussian random matrices. 
Subsequently, the author and Male solved the counterpart for Haar unitary matrices in \cite{MR3205602}, with no explicit speed of convergence. 
This result was refined further by Parraud \cite{2005.13834}
with explicit speeds of convergence, relying on ideas of  \cite{1912.04588}.  
The strongest result concerning strong convergence of random unitaries can be found in \cite{2012.08759}:

\begin{Th}\label{thm-bc2}
$(\overline U_i^{\otimes q_-}\otimes U_i^{\otimes q_+})_{i=1,\ldots , d}$ are strongly asymptotically free as $n\to\infty$ on
the orthogonal of fixed point spaces. 
\end{Th}
This means that strong asymptotic freeness does not hold at the sole level of the fundamental representation of $\Un$, but with respect to any sequence of representation associated with a
non-trivial $(\lambda, \mu )$. In other words, the only obstructions to strong freeness are the dimension one irreducible representations of $\Un$. 
We need a \emph{linearization step}, popularized by \cite{MR2183281} to evaluate the norm of 
$\sum_{i=-d}^d a_i\otimes X_i^{(n)},$
where $X_{-i}^{(n)}=X_{i}^{(n)*}$ and $X_{0}^{(n)}=Id$.
Although this first simplification step was sufficient to obtain strong convergence for iid $GUE$ -- i.e., matrices with
high symmetries -- thanks to analytic techniques, this turns out to be insufficient when one has to resort to moment methods. In \cite{MR4024563}, we initiated techniques based on a \emph{operator version} of \emph{non-backtracking
theory}, which we generalized in \cite{2012.08759}. We outline one key feature here. 

We consider $(b_1, \ldots, b_l)$  elements in $\mathcal B(\mathcal H)$ where $\mathcal H$ is a Hilbert space. 
We assume that the index set is endowed with an involution 
$i \mapsto i^{*}$ (and $i^{**} = i$ for all $i$). 
The \emph{non-backtracking operator} associated to the $\ell$-tuple of matrices 
$(b_1, \ldots, b_l)$ 
 is the operator on 
$\mathcal B(\mathcal H \otimes \mathbb{C}^l)$ defined by
\begin{equation}
B = \sum_{j \ne i^*} b_j  \otimes E_{ij},
\end{equation}
The following theorem allows to leverage moments techniques on linearization of non-commutative
polynomials through the study of $B$ :
\begin{Th}\label{th-non-backtracking}
Let  $\lambda \in \mathbb{C}$ satisfy 
$\lambda^2  \notin  \cup_{i \in \{1,\ldots , l\}} \mathrm{spec}(  b_i b_{i^*})$. 
Define the operator $A_\lambda$ on $\mathcal H$  through
 $$A_\lambda  = b_0 ( \lambda) +  \sum_{i=1}^\ell b_i (\lambda) \,, \qquad b_i (\lambda) = \lambda  b_{i} ( \lambda^2 - b_{i^*} b_{i} )^{-1}$$ 
and 
$$b_0 ( \lambda) =  - 1 -  \sum_{i=1}^\ell  b_{i}(  \lambda^2 - b_{i^*} b_i )^{-1} b_{i^*}.$$
Then $\lambda \in \sigma (B)$ if and only if  $0 \in \sigma ( A_\lambda)$. 
\end{Th}

\subsection{Centering and uniform Weingarten estimates}

To use Theorem \ref{th-non-backtracking}, one has to understand the spectral radius of the operator $B$ and, therefore, evaluate $\tau (B^TB^{*T})$
with $T$ growing with the matrix dimension, and this can be done through moment methods as soon as we
have uniform estimates on Weingarten functions. The first uniform estimate was obtained in  
\cite{MR3057186} and had powerful applications to the study of area laws in mathematical physics, 
however, it was not sufficient for norm estimates, and it was superseded by 
\cite{MR3680193}:
\begin{Th} 
For any $\sigma\in S_{k}$ and  
$n > \sqrt{6} k^{7/4}$, 
$$\frac{1}{1-\frac{k-1}{n^2}} \le 
\frac{n^{k+|\sigma|} \Wg(n,\sigma )}{\mathrm{Moeb}(\sigma)}  
\le \frac{1}{1- \frac{6 k^{7/2}}{n^2}}.$$
In addition, the l.h.s inequality is valid for any $n \ge k$.
\end{Th}
This result already enables us to prove Theorem \ref{thm-bc2} in the case where $q_-\ne q_+$ because there are no 
fixed points there. 
Let us now outline how to tackle the case $q_-= q_+$, which is interesting because it has fixed points. 
To handle fixed points, we need to introduce the \emph{centering} of a random variable $X$, namely $[X]=X-E(X)$.
For a symbol $\varepsilon \in \{ \cdot,-\}$ and $z \in \mathbb{C}$, we take the notation that $z^{\varepsilon}=z$ if $\varepsilon=\cdot$ and 
$z^{\varepsilon}=\overline z$ if
$\varepsilon = -$.
We want to to compute, for $U = (U_{ij})$ Haar distributed on $\Un$, expresssions of the form
$\E \prod_{t=1}^T \Big[  \prod_{l = 1}^{k_t} U_{x_{tl} y_{tl }}^{\varepsilon_{tl}}\Big]$
in a meaningful way.
We can write a Weingarten formula:
$$\E \prod_{t=1}^T [  \prod_{l = 1}^{k_t} U_{x_{tl} y_{tl }}^{\varepsilon_{tl}}]=
\sum_{\sigma,\tau\in P_2(k_1+\ldots +k_T)}\delta_{\sigma, x}\delta_{\tau ,y}\Wg(\sigma, \tau ; k_1,\ldots ,k_T),$$
where the function $\Wg$ depends on the pairings and the partition.
We say that a block of the partition $\{\{1,\ldots , k_1\},\ldots, \{k_1+\ldots +k_{T-1}+1,\ldots, 
k_1+\ldots +k_T\}\}$ is \emph{lonesome} with respect to the pairing $(\sigma, \tau)$ iff the group generated by
$\sigma, \tau$ stabilizes it. In \cite{2012.08759}, we prove
\begin{Th}
$\Wg$ decays as $n^{-k}$ where $k=(k_1+\ldots +k_T)/2+d(\sigma,\tau)+2\# \mathrm{lonesome\,\, blocks}$,
and this estimate is uniform on $k\sim Poly (n)$. 
\end{Th}
This Theorem, together with a comparison with Gaussian vectors, allows proving Theorem \ref{thm-bc2}.

\section{Perspectives}

Understanding better how to integrate over compact groups is a fascinating problem connected to many questions in various branches of mathematics and other scientific fields. 
We conclude this manuscript with a brief and completely subjective list of perspectives. 

\begin{enumerate}
\item
\emph{Uniform measures on  (quantum) symmetric spaces:}

Viewing a group as a compact manifold, can one extend the Weingarten calculus to other surfaces? 
Some substantial work has been done algebraically in this direction by Matsumoto \cite{MR3077830} 
in the case of symmetric spaces, see as well \cite{MR2408577} for the asymptotic version.
It would be interesting to study extensions of Matsumoto's results for compact quantum symmetric spaces. 

\item
\emph{Surfaces and Geometric group theory:}

An important observation by Magee and Puder is that if $G$ is a compact subgroup of $\Un$, the Haar measure
on $G^k$ yields a random representation of the free group $F_k$ on $\Un$ whose law is invariant under \emph{outer automorphisms} of $F_k$. This motivated them to compute, in \cite{MR4011702},
the expectation of the trace of non-trivial words in $(U_1,\ldots , U_k)\in \Un^k$. 
In addition to refining known asymptotics, they used the properties of the Weingarten
function to solve non-trivial problems about the orbits of  $F_k$ under the action by its outer conjugacy group.
In a different vein, Magee has very recently achieved a breakthrough by obtaining the first steps of Weingarten
calculus for representations of some one relator groups \cite{2101.00252,2101.03224}.

\item
\emph{Other applications to representation theory:}

The problem of calculating Weingarten functions on $SU(n)$ efficiently is complicated,
and even more so when the degree is high in comparison to $n$. A striking example is
$\int_{U\in SU(n)}\prod_{i,j=1}^n u_{ij}$. It was established in \cite{MR3322811}
that proving that this integral is non-zero
is equivalent to the Alon-Tarski conjecture.

More generally, this raises the question of computing efficiently integrals of high degree
on classical groups (typically, of degree $\ge n$ or $\ge n^2$).
Weingarten calculus, as developed in this manuscript, is not well adapted to this task. 
Some results in this direction have been obtained by Novak (\cite{2006.04304})
the author and Cioppa. 

\item
\emph{More tensors and norm estimates:}

In \cite{MR4024563,2012.08759}, we obtained strong convergence for an arbitrary finite number of tensors of random unitaries -- or random permutations. It turns out that the result can be relaxed a bit
to allow  the number 
of legs to vary slowly to infinity as the dimension of the group goes to infinity. 
This points to a double limit problem, and we wonder to which extent the number of legs of tensors and the size of the matrix can be independent. In the extreme case, could strong freeness hold for a given finite group but a number of tensors tending to infinity? 
Many variants of this problem exist, e.g., taking iid copies of unitaries instead of the same. 

Likewise, an important question is the behavior of $U_i\otimes U_j$ for arbitrary indices -- not only $i=j$ as in  \cite{MR4024563,2012.08759}. As observed by Hayes in \cite{2008.12287},
this is a possible approach towards the Peterson-Thom conjecture in operator algebras, and
it seems plausible that Weingarten calculus could help to solve this problem. 
\footnote{This manuscript was submitted to the ICM in the fall 2021. In the spring 2022, S. Belinschi and
M. Capitaine announced a proof of the Peterson-Thom conjecture in \cite{2205.07695}.}

\item
\emph{Maximizing functionals over groups:}

Given a polynomial function $f: G\to \mathbb{C}$, finding $m=\max_{U\in G} |f(U)|$ could provide approaches
to various conjectures in analysis in analysis or algebra. 
In general, finding the argmax is a problem intricately linked to the conjectures, and Haar integration could yield
 non-constructive approaches.
 Indeed, $l^{-1}\log\int (f(U)\overline{f(U)})^ld\mu_G(U)\sim 2\log m$, and the left hand side could in principle
 be approached with Weingarten calculus.
 Let us mention the example of the 
 Hadamard conjecture. It states that for any $4/n$, there exists an
 orthogonal matrix in $\M_n (\mathbb{R})$ whose entries are $\pm 1$. 
 An approach to this problem would be to show that the minimum of the polynomial function $f(U)=\sum_{ij}u_{ij}^4$ on $\mathcal{O}_n$ is $1$.
 We refer to \cite{MR2745424} for attempts with Weingarten calculus. 
 We also believe that some important problems in operator algebra could be approached that way (e.g., the problem of the non-existence of hyperlinear group).

\end{enumerate}

\section*{Acknowledgments}

It has always been highly stimulating to work with people from very diverse backgrounds:  
coauthors, graduate students, postdoctoral fellows. I want to thank
them all for our collaborations. 
I would like to thank Charles Bordenave, Mike Brannan, Luca Lionni, Sho Matsumoto, Akihiro Miyagawa, Ion Nechita
and
Jonathan Novak for reading my manuscript carefully and for many suggestions of improvements.

\section*{Funding}
This work was partially supported by JSPS Kakenhi 17H04823, 20K20882, 21H00987.
Most of the results presented in this note have been supported by JSPS Kakenhi, NSERC grants, and ANR grants, while the author was working at either of the following places: CNRS (Lyon 1), the University of Ottawa, Kyoto University.

\end{document}